\documentclass[12pt]{article}
\usepackage{hyperref}
\usepackage{enumerate,fullpage}
\usepackage{amssymb,amsmath,graphicx,amsthm}
\usepackage{color}

\usepackage{amsfonts,amsthm, amsmath}
\usepackage{epsfig}
\usepackage{makeidx}
\usepackage{graphicx,epstopdf}
\DeclareGraphicsExtensions{.ps,.png,.jpg}
\usepackage{enumerate}

\makeindex
\usepackage{subfig}

\newenvironment{namelist}[1]{%
\begin{list}{}
{

\settowidth{\labelwidth}{#1}
\setlength{\leftmargin}{1.1\labelwidth}
}
}{%
\end{list}}
\newcommand{\ncom}{\newcommand}
\ncom{\ul}{\underline}
\ncom{\beq}{\begin{equation}}
\ncom{\eeq}{\end{equation}}
\ncom{\bea}{\begin{eqnarray*}}
\ncom{\eea}{\end{eqnarray*}}
\ncom{\beqa}{\begin{eqnarray}}
\ncom{\eeqa}{\end{eqnarray}}
\ncom{\nno}{\nonumber}
\ncom{\non}{\nonumber}
\ncom{\ds}{\displaystyle}
\ncom{\half}{\frac{1}{2}}
\ncom{\mbx}{\makebox{.25cm}}
\ncom{\hs}{\mbox{\hspace{.25cm}}}
\ncom{\rar}{\rightarrow}
\ncom{\Rar}{\Rightarrow}
\ncom{\noin}{\noindent}
\ncom{\bc}{\begin{center}}
\ncom{\ec}{\end{center}}
\ncom{\sz}{\scriptsize}
\ncom{\rf}{\ref}
\ncom{\s}{\sqrt{2}}
\ncom{\sgm}{\sigma}
\ncom{\Sgm}{\Sigma}
\ncom{\psgm}{\sigma^{\prime}}
\ncom{\dt}{\delta}
\ncom{\Dt}{\Delta}
\ncom{\lmd}{\lambda}
\ncom{\Lmd}{\Lambda}
\ncom{\Th}{\Theta}
\ncom{\e}{\eta}
\ncom{\eps}{\epsilon}
\ncom{\pcc}{\stackrel{P}{>}}
\ncom{\lp}{\stackrel{L_{p}}{>}}
\ncom{\dist}{{\rm\,dist}}
\ncom{\sspan}{{\rm\,span}}
\ncom{\re}{{\rm Re\,}}
\ncom{\im}{{\rm Im\,}}
\ncom{\sgn}{{\rm sgn\,}}
\ncom{\ba}{\begin{array}}
\ncom{\ea}{\end{array}}
\ncom{\hone}{\mbox{\hspace{1em}}}
\ncom{\htwo}{\mbox{\hspace{2em}}}
\ncom{\hthree}{\mbox{\hspace{3em}}}
\ncom{\hfour}{\mbox{\hspace{4em}}}
\ncom{\vone}{\vskip 2ex}
\ncom{\vtwo}{\vskip 4ex}
\ncom{\vonee}{\vskip 1.5ex}
\ncom{\vthree}{\vskip 6ex}
\ncom{\vfour}{\vspace*{8ex}}
\ncom{\norm}{\|\;\;\|}
\ncom{\integ}[4]{\int_{#1}^{#2}\,{#3}\,d{#4}}
\ncom{\vspan}[1]{{{\rm\,span}\{ #1 \}}}
\ncom{\dm}[1]{ {\displaystyle{#1} } }
\ncom{\ri}[1]{{#1} \index{#1}}

\newtheorem{remark}{\bf Remark}[section]
\newtheorem{proposition}{Proposition}[section]

\newtheoremstyle
    {remarkstyle}
    {}
    {11pt}
    {}
    {}
    {\bfseries}
    {:}
    {     }
    {\thmname{#1} \thmnumber{#2} }

\theoremstyle{remarkstyle}

\begin{document}

\newpage

\begin{center}
{\Large \bf Tempered Fractional Poisson Processes and Fractional Equations with $Z$-Transform}
\end{center}
\vone
\begin{center}
{Neha Gupta}$^{\textrm{a}}$, {Arun Kumar}$^{\textrm{a}}$, {Nikolai Leonenko}$^{\textrm{b}}$

\footnotesize{
		$$\begin{tabular}{l}
		$^{\textrm{a}}$ \emph{Department of Mathematics, Indian Institute of Technology Ropar, Rupnagar, Punjab - 140001, India}\\
		
$^{b}$Cardiff School of Mathematics, Cardiff University, Senghennydd Road,
Cardiff, CF24 4AG, UK

\end{tabular}$$}
\end{center}
\vtwo

\vtwo
\begin{center}
\noindent{\bf Abstract}
\end{center}
In this article, we derive the state probabilities of different type of space- and time-fractional Poisson processes using $z$-transform. We work on tempered versions of time-fractional Poisson process and space-fractional Poisson processes. We also introduce Gegenbauer type fractional differential equations and their solutions using $z$-transform. Our results generalize and complement the result available on fractional Poisson processes in several directions.
\section{Introduction}
In recent years fractional processes are getting increased attention due to their real life applications. For example fractional Brownian motion (FBM) overcome the limitations of Brownian motion in modeling of long-range dependent phenomena occurring in financial time series, Nile river data and fractal analysis etc (see e.g. Beran, 1994). Similarly time-fractional Poisson process is helpful in modeling of counting processes where the inter-arrival times are heavy tailed or arrivals are delayed (see e.g. Meerschaert et al. 2011; Laskin, 2003). In time-fractional Poisson process the waiting times are Mittag-Leffler (ML) distributed see Laskin(2003). Recently, Orshinger and Polito (2009) introduced space-fractional Poisson process by taking a fractional shift operator in place of an integer shift operator in the governing differential-difference equation of standard Poisson process. Moreover, they have shown that space-fractional Poisson process can also be obtained by time-changing the standard Poisson process with a stable subordinator. Further, they argue that time-fractional Poisson process and the space-fractional Poisson process are specific cases of the same generalized complete model and hence might be useful in the study of transport of charge carriers in semiconductors (Uchaikin and Sibatov, 2008) or applications related to fractional quantum mechanics (Laskin, 2009). In this article, we extend the space-fractional and time-fractional Poisson process by considering a tempered time-space-fractional Poisson process. We feel a strong motivation to study these processes since tempering introduces a finite moment condition in space-fractional Poisson process. Further, it gives more flexibility in modeling of natural phenomena discussed in Laskin (2009), due to extra parameters which can be picked based on the situation. Moreover, we suggest to use $z$-transform since the z-transform method is more general then method of probability generating functions, and hence it could be applied for solutions of fractional equations which are not probability distributions, see Section 5.6. The  governing  equations  for  marginal  distributions  of  Poisson and  Skellam  processes  time-changed  by  inverse  subordinators are discussed in Buchak and Sakhno (2018a). For properties of Poisson processes directed by compound Poisson-Gamma subordinators see Buchak and Sakhno (2018b).\\

The rest of the paper is organized as follows. In Section  2, we introduce $z$-transform and inverse $z$-transform also indicate their main characteristics. Caputo-Djrbashian fractional derivative is discussed in Section 3. In Section  4, main properties of Poisson process are discussed briefly. Section  5 is devoted to different kind of fractional Poisson processes. In this section first we revisit the time- and space-fractional Poisson processes with $z$-transform approach. Our main results are given in  Sections 5.4, 5.5 and 5.6. The last section concludes.

\section{$z$-transform and Inverse $z$-transform}
The $z$-transform is a linear transformation and can be considered as an operator mapping sequence of scalars into functions of complex variable $z$. For a function $f(k), k\in \mathbb{Z}$, the bilateral $z$-transform is defined by
 \begin{eqnarray*}
F(z) = \mathcal{Z}{f(k)}= \sum_{k= -\infty}^{\infty} f(k) z^{-k}, z\in\mathbb{C}
 \end{eqnarray*}
We assume that there exists an $R$ (radius of convergence) such that series converges for $|z|>R$.
The  inverse $z$-transform is defined by the  complex integral
\begin{align*}
 \mathcal{Z}^{-1}\{F(z)\}= f(k) =\frac{1}{2\pi i}\oint\limits_{C}F(z)z^{k-1}\,\mathrm{d}z,
\end{align*}
where $C$ is simple closed contour enclosing the origin and lying outside the circle $|z|=R$. The existence of the inverse imposes restrictions on $f(k)$ for the uniqueness.
Alternatively, in case where $f(k)$ is defined only for $k\in \mathbb{N}$, the (unilateral) $z$-transform is defined as.
\begin{eqnarray*}
F(z) = \mathcal{Z}{f(k)}= \sum_{k= 0}^{\infty} f(k) z^{-k}, z\in\mathbb{C}
 \end{eqnarray*}
\begin{align*}
F(z)=f(0)+f(1)z^{-1}+f(2)z^{-2}+\ldots+f(k)z^{-k}+\ldots,
\end{align*}
where the coefficient of $z^{-k}$ in this expansion is the inverse given by
\begin{align*}
f(k)=\mathcal{Z}^{-1}(F(z)).
\end{align*}
If $f(k)$, $k\in\mathbb{N}$, is probability distribution, that is
\begin{align*}
f(k)\geq 0,  \sum_{k= 0}^{\infty} f(k)=1,
\end{align*}
Then the probability generating function (PGF) is defined as
\begin{align*}
G(s)=\sum_{k= 0}^{\infty}s^{k}f(k), |s|\leq1,
\end{align*}
and relates to unilateral  $z$-transform as follows $G(Z^{-1})=F(z)$\\
The following operational properties of $z$-transform are used further for the solution of initial value problem involving difference equations
\begin{align*}
 \mathcal{Z}{f(k)}=F(z),   
\end{align*}
\begin{equation}\label{property of ztrans}  
 \mathcal{Z}(f(k-m))=z^{-m}[F(z)+\sum_{r=-m}^{-1}{f(r)z^{-r}}],\; m\geq 0,
\end{equation}
\begin{equation}
 \mathcal{Z}(f(k+m))=z^{m}[F(z)-\sum_{r=0}^{m-1}{f(r)z^{-r}}].
\end{equation}

\section{Caputo-Djrbashian Fractional Derivative}
The Caputo-Djrbashain (CD) fractional derivative of a function $u(t),t\geq0$ of order $\beta\in(0,1],$ is defined as

\begin{equation}\label{Caputo_Djrbashian_Derivative}
\frac{d^{\beta}u}{dt^{\beta}}=D^{\beta}_{t}u(t)= \frac{1}{\Gamma{(1-\beta)}}\int_{0}^{t}\frac{du(\tau)}{d{\tau}}\frac{d{\tau}}{(t-\tau)^{\beta}},\; \beta\in(0,1].
\end{equation}

Note that the classes of functions for which the CD derivative is well defined is discussed in [Meerschaert and Sikorski (2012), Sections 2.2, 2.3].\\
The Laplace transform (LT) of CD fractional derivative is given by
$$
\int_{0}^{\infty}e^{-st}D_{t}^{\beta}u(t)dt=s^{\beta}F(s)-s^{\beta-1}u(0^{+}),\; 0<\beta<1,
$$
Where F(s) is the LT of the function $u(t),t\geq0$, see [Meerschaert and Sikorski (2012), p.39]
$$
F(s)=\int_{0}^{\infty}e^{-st}u(t)dt, Re(s)>0.
$$

\noindent  The Laplace transform (LT) of Caputo fractional derivative is given by
\begin{equation}\label{LT-Caputo}
\mathcal{L}\left(\frac{d^{\beta}}{dx^{\beta}} f(x)\right) = s^{\beta} F(s) - \sum_{k=0}^{n-1}s^{\beta-k-1} f^{(k)}(0),
\end{equation}
where $\mathcal{L}(f(x)) = F(s)$.

\section{Poisson Process}
The homogeneous Poisson process $N(t), t\geq0$, with parameter $\lambda>0$ is defined as,
\begin{align*}
N(t)= \max\{n: T_1+T_2+ \ldots +T_n\geq0\},\; t\geq 0,
\end{align*}
 where the inter-arrival times $T_{1},T_2,\ldots$ are non-negative iid exponential random variable with mean $1/\lambda$. The probability mass function (PMF) $P(k,t)= \mathbb{P}(N(t)=k)$ is given by
\begin{equation}
P(k,t) = \mathbb{P}(N(t)=k)= \frac{e^{-\lambda t}(\lambda t)^k}{k!},\;k=0,1,\ldots
\end{equation}
The PMF of the Poisson process govern the following differential-difference equation
\begin{equation}\label{DDE-PP}
\frac{d}{dt}P(k,t)= -\lambda (P(k,t) - P(k-1,t)) = -\lambda \nabla P(k,t)
\end{equation}
\begin{equation}\label{intial condition}
P(k,0) =\delta_{k,0}=
          \begin{cases}
                  0, & k\neq0,\\
                  1, & k=0.
           \end{cases}
\end{equation}
Where  by the definition $P(-1,t)=0$, and $\nabla = (1-B)$ with $B$ as backward shift operator, that is  
$B\{P(k,t)\}=P(k-1,t)$. 
\section{Fractional Poisson Process}
In this section, we revisit space- and time-fractional Poisson processes using the $z$-transform approach. Note that $z$-transform is more general than the probability generating function approach and can be used to solve the difference-differential equations where the solution may not be a probability distribution. Also, we introduce and study tempered space-time-fractional Poisson processes. Further, to show the importance of $z$-transform, we consider Gegenbauer type fractional difference equations.
\subsection{The Time-Fractional Poisson Process}
The time-fractional Poisson process(TFPP) was first introduced by (see Mainardi et al , 2004) as renewal process
\begin{align}
N_{\beta}(t)= \max\{n: T^{\beta}_{1}+ \ldots +T^{\beta}_{n}\leq0\},\; t\geq0, \beta\in(0,1],
\end{align}
where the inter arrival times $T^{\beta}_{1},T^{\beta}_{2}\ldots$ are iid non-negative random variables with Mittag-Leffler distribution function.
\begin{align}
 \mathbb{P}(T^{\beta}_{k}\leq x)=1-E_{\beta}({-\lambda} x^{\beta}),x\geq0,\; \lambda>0.
\end{align}
Probability distribution function (PDF)
\begin{align}
f(x)={\lambda}x^{\beta-1}E_{\beta,\beta}({-\lambda}x^{\beta}),x\geq0,>0,\lambda>0, \beta\in(0,1],	
\end{align}
where
\begin{align}
E_{a,b}(z)=\sum_{k= 0}^{\infty}\frac{z^{k}}{\Gamma(ak+b)},z\in \mathbb{C},a,b\geq0,
\end{align}
is two parametric Mittag-Leffler distribution function, and $E_{a}(z)=E_{a,1}(z),z\in \mathbb{C}$ is the classical Mittag-Leffler function (see Gorenflo et al. 2014).\\ 
Let $S_{\beta}(t),t\geq0$ be a stable subordinator with Laplace transform\\
\begin{align}
\mathbb{E}e^{-zS_{\beta}(t)}=e^{-tz^\beta},\; z>0,\; t\geq0,\; \beta\in(0,1),	
\end{align}
We define it's right-inverse process called inverse stable subordinator, as
\begin{align}
Y_{\beta}(t)= \inf\{w>0: S_{\beta}(w)> t\}, \;t\geq 0.
\end{align}
The process $Y_{\beta}(t), \;t\geq0$ is non-Markovian with non-stationary increment (see Bingham 1971).\\
Alternatively, Meerschaert et al. (2011) find the following subordinator representation of TFPP: 
\begin{align}
N_{\beta}(t)=N(Y_{\beta}(t)),\; t\geq0,\; \beta\in(0,1),
\end{align}
where $N(t),t\geq0$ is the homogenous Poisson process with parameter $\lambda>0$ and $Y_{\beta}(t),t\geq0$, is independent of $N(t),t\geq0$. 
Beghin and Orshinger (2009) have shown that the PMF
\begin{align}
P_{\beta}(k,t)=\mathbb{P}\{N_{\beta}(t)=k\}=\frac{(\lambda t^{\beta})^k}{k!}\sum_{r=0}^{\infty}\frac{(k+r)!}{r!}\frac{(-\lambda t^{\beta})^r}{\Gamma(\beta(k+r)+1)},\; k=0,1,2 \ldots,\;  \beta\in(0,1].
\end{align}
And it is solution of the following functional differential-difference equation with CD fractional derivative in time:
\begin{align}\label{derivative in time}
\frac{d^{\beta}}{dt^{\beta}}P_{\beta}(k,t) = -\lambda^{\beta} (P_{\beta}(k,t) - P_{\beta}(k-1,t)) = -\lambda^{\beta} \nabla P_{\beta}(k,t),\\
  P_{\beta}(k,t)=0,\;\; \mathrm{where}\;\; k<0, \\
 P_{\beta}(k,0)=\delta_{k,0}, \; k=0,1,2 \ldots ,\;  \beta\in(0,1].
\end{align}
For the particular case $\beta=1$, the process reduce to the standard Poisson process. 

\subsection{The Space-Fractional Poisson Process}
Let $S_{\alpha}(t),t\geq0, \alpha\in(0,1)$, be a stable subordinator and $N(t),t\geq0$, is homogenous Poisson process with parameter $\lambda>0$, independent of $S_{\alpha}(t),t\geq0$.\\
The space-fractional Poisson process (SFPP) $N^{\alpha}(t),t\geq0,0<\alpha<1$ was introduced by Orshinger and Polito (2012) as follows
\begin{equation}
 N^{\alpha}(t) =
          \begin{cases}
                  N(S_{\alpha}(t)),t\geq0 & 0<\alpha<1,\\
                  N(t),t\geq0, & \alpha=1.
           \end{cases}
\end{equation}
The  density
$f(x,1)$ of $S_{\alpha}(1)$ is infinitely differentiable on $(0,\infty)$,
with the asymptotics as follows (see Uchaikin and Zolotarev, 1999):
\begin{equation}\label{D-asymptotic-0}
f(x,1)\sim
\frac{(\frac{\alpha}{x})^{\frac{2-\alpha}{2(1-\alpha)}}}{\sqrt{2\pi\alpha(1-\alpha)}}
e^{-(1-\alpha)(\frac{x}{\alpha})^{-\frac{\alpha}{1-\alpha}}}, \ \
\mathrm{as}\ \ x\to 0;
\end{equation}
\begin{equation}\label{D-asymptotic-large}
f(x,1)\sim \frac{\alpha}{\Gamma(1-\alpha)x^{1+\alpha}}, \ \
\mathrm{as} \ \ x\to \infty.
\end{equation}
Exact form of the density $f(x,1)$ in term of infinite series or integral are discussed in [Aletti et al. 2018;  Kumar and Vellaisamy, 2015] and has the following infinite-series form
\begin{align} \label{stable-series-form}
f(x,t) = 
      \displaystyle\frac{1}{\pi} \sum_{k=1}^{\infty}(-1)^{k+1}\frac{\Gamma(k\alpha +1)}{k!}\frac{1}{x^{k\alpha+1}}\sin\left(\pi\alpha k\right),\; x>0.
\end{align}
Note that from \eqref{D-asymptotic-0} and
\eqref{D-asymptotic-large}, we have
\begin{equation}\label{asymp_stable_density}
\lim_{x\rightarrow 0} f(x,1)= f(0,1)=0~~and~~\lim_{x\rightarrow
\infty} f(x,1)= f(\infty,1)=0.
\end{equation}
The probability generating function (PGF) of this process is of the form.\\
\begin{align}\label{PGF of space}
G^{\alpha}(s,t)=\mathbb{E}s^{N^{\alpha}(t)}=e^{{-\lambda}^{\alpha}(1-s)^{\alpha}t},\; |s|\leq1, \; \alpha\in(0,1).
\end{align}
We introduce the fractional difference operator (see Beran, 1994, p-60)
\begin{align}\label{difference opr}
\nabla ^\alpha=(1-B)^\alpha=\sum_{k=0}^{\infty}{\alpha \choose k}(-1)^{k}B^{k}, \;\; \alpha\in(0,1),
\end{align}
where
$$
{\alpha \choose k}=\frac{(\alpha)(\alpha-1)\ldots(\alpha-k+1)}{k!}=\frac{(-1)^{k}(-\alpha)_{k}}{k!},
$$
 and  Pochhammer symbol
\begin{align}\label{Pochhammer}
(\lambda)_{k}=
          \begin{cases}
                  \lambda(\lambda+1)\ldots(\lambda+k-1), & k=1,2, \ldots \\
                  1, & k=0.
           \end{cases}
\end{align}
Let 
\begin{align} 
P^{\alpha}(k,t) =\mathbb{P}\{N^{\alpha}(t)=k\},   k=0,1, \ldots.
\end{align}
The PMF $P^{\alpha}(k,t)$ satisfies the following fractional differential-difference equations (see e.g. Orsingher and Polito, 2012).
\begin{align}\label{SFPP}
 \frac{d}{dt}P^{\alpha}(k,t) &= -\lambda^\alpha (1-B)^\alpha P^{\alpha}(k,t),\;\;  \alpha\in(0,1],\; k=1,2,\ldots \\
\frac{d}{dt}P^{\alpha}(0,t)& = -\lambda^\alpha P^{\alpha}(0,t),
\end{align}
with initial conditions            
\begin{equation}\label{initial-conditions}  
  P^{\alpha}(k,0) = \delta_{k,0}.
\end{equation}
\noindent Using the $z$-transform in both side, it follows
\begin{align*}
 \frac{d}{dt}\{\mathcal{Z}{P^{\alpha}(k,t)}\} = -\lambda^\alpha[\mathcal{Z}\{{(1-B)^\alpha P^{\alpha}(k,t)\}}].
\end{align*}
Further, using \eqref{difference opr}, we have
\begin{align*}
 \frac{d}{dt}\{\mathcal{Z}{P^{\alpha}(k,t)}\}=-\lambda^\alpha\mathcal{Z}P^{\alpha}(k,t)\left[1-\frac{\alpha}{z}+\frac{\alpha(\alpha-1)}{(2!)(z^2)}-\ldots \right].
\end{align*}
Further,
\begin{align}\label{intermediate-eq}
 \frac{d}{dt}\{\mathcal{Z}{P^{\alpha}(k,t)}\}=\left[-\lambda\left(1-\frac{1}{z}\right)\right]^{\alpha} \mathcal{ Z}{P^{\alpha}(k,t)}.
\end{align}
Solving \eqref{intermediate-eq} for $\mathcal{Z}{P^{\alpha}(k,t)}$ and using initial conditions in \eqref{initial-conditions}, leads to
\begin{align*}
 \mathcal{Z}{P^{\alpha}(k,t)}&=e^{-\lambda^{\alpha}(1-\frac{1}{z})^\alpha t}
=\sum_{r=0}^{\infty}(-{\lambda}^{\alpha})^{r} t^{r}(1-z^{-1})^{\alpha r}=\sum_{r=0}^{\infty}(-{\lambda}^{\alpha})^{r} t^{r}\sum_{k=0}^{\infty}(-1)^{k}{\alpha r \choose k} {z}^{-k}\\
&=\sum_{k=0}^{\infty}z^{-k}\left[\frac{(-1)^k}{k!}\sum_{r=0}^{\infty}\frac{(-{\lambda}^{\alpha})^{r} t^r}{r!}\frac{\Gamma(r\alpha+1)}{\Gamma(r\alpha-k+1)}\right].
\end{align*}
To find $P^{\alpha}(k,t)$, invert the $z$-transform that is equivalent to find the coefficient of $z^{-k}$, which leads to
\begin{align}\label{space-fractional-PMF}
P^{\alpha}(k,t) =\frac{(-1)^k}{k!}\sum_{r=0}^{\infty}\frac{(-\lambda^\alpha)^{r} t^r}{r!}\frac{\Gamma(r\alpha+1)}{\Gamma(r\alpha-k+1)}.
\end{align}
Moreover, one can write 
\begin{equation}
\left(1-\frac{1}{z}\right)^{2\alpha} = \left(1 - \frac{1}{z}\left(2 - \frac{1}{z}\right)\right)^{\alpha}.
\end{equation}
By comparing the coefficients of $z^{-2p}$ in both side, we get the following identity.
For any $\alpha>0$ and $p\in\mathbb{N}$, the following identify holds
\begin{align}\label{Binomial-Identity}
{2\alpha \choose 2p} &= {\alpha \choose p}\cdot {p\choose 0} + 2^2{\alpha \choose p+1}\cdot {p+1\choose 2} + \ldots \nonumber\\
&+ 2^{2p-2}{\alpha \choose 2p-1}\cdot {2p-1\choose 2p-2} + 2^{2p}{\alpha \choose 2p}\cdot {2p\choose 2p}, 
\end{align}
where ${\alpha \choose p} = \frac{\Gamma(\alpha+1)}{\Gamma(p+1)\Gamma(\alpha-p+1)}.$ The identity \eqref{Binomial-Identity} is used in subsequent section.
\begin{remark}
The composition of $n$ stable subordinators is also a stable subordinator.  Let $S_1, S_2,\ldots, S_n$ be $n$ independent stable subordinators with parameters $\alpha_i,\;i=1,2\ldots,n$.
Then the iterated composition defined by
$S^{(n)}(t)=S_1oS_2o\cdots oS_n(t)$ is also a stable subordinator with parameter $\alpha_1\alpha_2\cdots\alpha_n.$ It is easy to show that the PMF
$\tilde{P}(k,t)= \mathbb{P}(N(S^{(n)}(t))=k)$, satisfies the following
equation
\begin{equation}
 \frac{d^{2^n}}{dt^{2^n}}\tilde{P}(k, t) = \lambda
[\tilde{P}(k, t) - \tilde{P}(k-1, t)]. 
\end{equation}
\end{remark}

\subsection{The Time-Space-Fractional Poisson Process}
Note that Orshinger and Polito (2012) introduced the following time-space-fractional differential equations
\begin{align}\label{STFPP}
\frac{d^\beta}{dt^\beta}P^{\alpha}_{\beta}(k,t) &= -\lambda^\alpha(1- B)^\alpha P^{\alpha}_{\beta}(k,t),  \;\; \alpha\in(0,1] \;\;  \beta\in(0,1),\; k=1,2,\ldots \\
\frac{d^{\beta}}{dt^{\beta}}P^\alpha_{\beta}(0,t)& =-\lambda^\alpha P^\alpha_{\beta}(0,t),
\end{align}
with initial conditions
\begin{equation}\label{initial-conditions2}  
  P^{\alpha}_{\beta}(k,0) =\delta_{k,0}=
          \begin{cases}
                  0, & k>0,\\
                  1, & k=0,
           \end{cases}
\end{equation}
where $\frac{d^\beta}{dt^\beta}$ is the Caputo-Djrbashian derivative defined in \eqref{Caputo_Djrbashian_Derivative}.
They have shown that 
\begin{align}\label{TSPP} 
P^{\alpha}_{\beta}(k,t)=\frac{(-1)^k}{k!}\sum_{r=0}^{\infty}\frac{(-\lambda^\alpha)^{r}t^{r\beta}}{\Gamma(1 + r\beta)}\frac{\Gamma(r\alpha+1)}{\Gamma(r\alpha-k+1)},\;\; k=0,1,\ldots,
\end{align}
and its probability generating function (PGF)
\begin{align}
G_{\beta}^{\alpha}(u,t)=\sum_{k=0}^{\infty}u^{k}P_{\beta}^{\alpha}(k,t)=E_{\beta}\left(-\lambda^{\alpha}t^{\beta}(1-u)^{\alpha}\right),\; |u|\leq 1.
\end{align}
Using $z$-transform, we can present as alternative proof of the fact that \eqref{TSPP} govern the equation
To solve \eqref{TSPP}, take the $z$-transform in both hand side, leads to 
\begin{align*}
 \frac{d^\beta}{dt^\beta} \{ \mathcal{Z} P_{\beta}^{\alpha}(k,t)\}= -\lambda^\alpha(1-z^{-1})^\alpha \{ \mathcal{Z} P_{\beta}^{\alpha}(k,t) \}.
\end{align*}
Further, using the Laplace transform with respect to the time variable $t$ and $\mathcal{Z}\{P^\alpha_\beta(k,0) \} =1$, it follows
\begin{align*}
 s^\beta \mathcal{L}[ \mathcal{Z} \{ P_{\beta}^{\alpha}(k,t) \} ]- s^{\beta-1}= -\lambda^\alpha(1-z^{-1})^\alpha\mathcal{L}[\mathcal{Z}\{ P_{\beta}^{\alpha}(k,t)\}].
\end{align*}
By some manipulation, it follows
\begin{align*}
\mathcal{L}[\mathcal{Z}\{P_{\alpha}(k,t)\}] = \frac{s^{\beta-1}}{s^\beta+ \lambda^\alpha(1-z^{-1})^\alpha}.
\end{align*}
Using the LT of Mittag-Leffler function $\mathcal{L}(E_{\beta,1}(- u t^{\beta})) = \frac{s^{\beta-1}}{u +s^{\beta}}$ (see e.g. Meerscharet and Sikorski (2012), p. 36), it follows
\begin{align}\label{z-transform-TSFPP}
 \mathcal{Z}\{P^{\alpha}_{\beta}(k,t)\}&= \mathcal{L}^{-1}\left\{\frac{s^{\beta-1}}{s^\beta+\lambda^\alpha(1-z^{-1})^\alpha}\right\}
  = E_{\beta,1}((- \lambda^\alpha(1-z^{-1})^{\alpha} t^{\beta})\\
 & = \sum_{k=0}^{\infty}\frac{(-1)^k\lambda^{k\alpha}t^{k\beta}(1-z^{-1})^{k\alpha}}{\Gamma(1+k\beta)}\nonumber.
 \end{align}
 Inverting the $z$-transform gives
 \begin{align*}
 P^{\alpha}_{\beta}(k,t)&=\frac{(-1)^k}{k!}\sum_{r=0}^{\infty}\frac{(-\lambda^\alpha)^{r}t^{r\beta}}{\Gamma(1+ r\beta)}(r\alpha)(r\alpha-1) \ldots (r\alpha-k+1)\\
&=\frac{(-1)^k}{k!}\sum_{r=0}^{\infty}\frac{(-\lambda^\alpha)^{r}t^{r\beta}}{\Gamma(1 + r\beta)}\frac{\Gamma(r\alpha+1)}{\Gamma(r\alpha-k+1)},\;\; k=0,1,2, \ldots
\end{align*}
Alternatively, one can define the time-space-fractional Poisson process (TSFPP) as follows
\begin{equation}\label{TSFPP-subordination}
N^\alpha_{\beta}(t)=N(S_{\alpha}(Y_{\beta}(t)) = N^{\alpha}(Y_{\beta}(t)),\; t\geq0,
\end{equation}
where TSFPP is obtained by subordinating the standard Poisson process $\{N(t)\}_{t\geq0}$ by an independent $\alpha$-stable subordinator $\{S_{\alpha}(t)\}_{t\geq0}$ and then by the inverse $\beta$-stable subordinator $\{Y_{\beta}(t)\}_{t\geq0}$. 
\begin{proposition}\label{prop5.1}
The state probabilities of time-space-fractional Poisson process defined in \eqref{TSFPP-subordination} satisfies the equation \ref{STFPP}. 
\end{proposition}
\begin{proof}
Using $z$-transform, it follows
\begin{align}\label{z-transform-TSFPP-alternative}
G(z,t)&= \sum_{k=0}^{\infty}z^{-k}\mathbb{P}(N^\alpha_{\beta}(t) =k) = \sum_{k=0}^{\infty}z^{-k}\mathbb{P}\left(N(S_{\alpha}(Y_{\beta}(t)))=k\right)\nonumber\\
& = \sum_{k=0}^{\infty}z^{-k}\mathbb{P}\left(N(Y_{\beta}(t)) = k\right) = \mathbb{E}\left[\sum_{k=0}^{\infty}z^{-k}\mathbb{P}\left(N(Y_{\beta}(t)) = k|Y_\beta(t)\right)\right]\nonumber\\
& = \mathbb{E}\left[e^{-\lambda^{\alpha}\left(1-z^{-1}\right)^{\alpha}Y_{\beta}(t)}\right] = E_{\beta,1}\left(-\lambda^{\alpha}\left(1-z^{-1}\right)^{\alpha}t^{\beta}\right),
\end{align}
which follows using the result $\mathbb{E}(e^{-sY_{\beta}(t)}) = E_{\beta,1}(-st^{\beta}).$
Note that the two $z$-transforms given in \eqref{z-transform-TSFPP} and \eqref{z-transform-TSFPP-alternative} are same and hence two representations are equivalent by the uniqueness of $z$-transform.
\end{proof}

In next subsections, we generalize the above discussed processes to their tempered counterparts, which can give more flexibility in modeling of the natural phenomena suggested for the space- and time-fractional Poisson processes due the the extra control parameter.

\subsection{The Tempered Space-Fractional Poisson Process}
One can also define tempered space-fractional Poisson process by subordinating homogeneous Poisson process with the tempered stable subordinator. Note that tempered stable subordinators are obtained by exponential tempering in the distribution of stable subordinator, see Rosinski (2007) for more details on tempering stable processes. Let
$f(x,t)$, $0<\alpha<1$ denotes the density of a stable subordinator $S_{\alpha}(t)$ with
LT
\begin{equation} 
 \int_{0}^{\infty}e^{-sx}f(x,t)
dx= e^{-ts^{\alpha}}. 
\end{equation}
A tempered stable subordinator
$S_{\alpha,\mu}(t)$ has a density

\begin{equation}\label{ts-density}
f_{\mu}(x,t)= e^{-\mu x+\mu^{\alpha}t}
f(x,t),~~ \mu>0. 
\end{equation}
Using \eqref{ts-density} and \eqref{asymp_stable_density}, it follows
\begin{equation}\label{asymp_temstable_density}
\lim_{x\rightarrow 0} f_{\mu}(x,t)= f_{\mu}(0,t)=0~~\mathrm{and}~~\lim_{x\rightarrow
\infty} f_{\mu}(x,t)= f_{\mu}(\infty,t)=0.
\end{equation}
The sample paths of $S_{\alpha,\mu}(t)$ are strictly increasing similar to the stable subordinator. Further the LT
\begin{equation}\label{tempered-LT}
\tilde{f}_{\mu}(s,t)=\int_{0}^{\infty}e^{-s x}f_{\mu}(x,t)dx =
e^{-t((s+\mu)^{\alpha}-\mu^{\alpha})}.
\end{equation}
For $S_{\alpha, \mu}(t)$, we have $\mathbb{E}(S_{\alpha, \mu}(t)) = \alpha \mu^{\alpha-1}t$ and $\mbox{Var}(S_{\alpha, \mu}(t)) = \alpha(1-\alpha) \mu^{\alpha-2}t$.
The tempered space-fractional Poisson process is defined by 
\begin{align}\label{tempered-SFPP-Subodination}
N^{\alpha,\mu}(t) := N(S_{\alpha,\mu}(t)),\;\alpha\in(0,1),\;\mu\geq 0,
\end{align}
where homogeneous Poisson process $N(t)$ is independent of $S_{\alpha,\mu}(t)$ a tempered stable subordinator. The tempered space-fractional Poisson process is a L\'evy process with finite integer order moments due to the finite moments of the tempered stable subordinators. However the integer order moments of space-fractional Poisson process are not finite.
The tempered space-fractional Poisson process with marginal PMF $P^{\alpha, \mu}(k,t)$ can also be defined by taking a tempered fractional shift operator instead of an ordinary fractional shift operator in \eqref{DDE-PP} such that

\begin{align}\label{tempered-SFPP}
\frac{d}{dt}P^{\alpha, \mu}(k,t) = - ((\mu + \lambda(1-B))^{\alpha} - \mu^{\alpha}) P^{\alpha, \mu}(k,t),  \;\; \alpha\in(0,1], \;\;  \mu\geq 0,
\end{align}
which reduces to the SFPP by taking $\mu =0.$ We have following proposition for the state probabilities of tempered space-fractional Poisson process. 
\begin{proposition}
The state probabilities for tempered space-fractional Poisson process are given by
\begin{equation}
P^{\alpha, \mu}(k,t) = (-1)^k\sum_{m=0}^{\infty}\mu^m\lambda^{-m}\sum_{r=0}^{\infty}\frac{(-t)^r}{r!}{\alpha r \choose m} {\alpha r -m \choose k},\;k=0,1,\ldots,\;\mu\geq 0,\;t\geq 0.
\end{equation}
For $\mu =0$, which reduces to \eqref{space-fractional-PMF}.
\end{proposition}
\begin{proof}
Suppose $G(z,t)$ is the $z$-transform of $P^{\alpha, \mu}(k,t)$, then
\begin{align*}
G(z,t) &= e^{-t \left(\left(\mu + \lambda\left(1-\frac{1}{z}\right)\right)^{\alpha} - \mu^{\alpha}\right)}\\
& = e^{t\mu^{\alpha}}\sum_{r=0}^{\infty}\frac{(-t)^r}{r!}\left(\mu + \lambda\left(1-\frac{1}{z}\right)\right)^{\alpha r}\\
& = e^{t\mu^{\alpha}} \sum_{r=0}^{\infty}\frac{(-t)^r}{r!}\sum_{m=0}^{\infty}{\alpha r \choose m} \mu^{m} \lambda^{\alpha r -m} \left(1-\frac{1}{z}\right)^{\alpha r - m}\\
& = e^{t\mu^{\alpha}} \sum_{r=0}^{\infty}\frac{(-t)^r}{r!}\sum_{m=0}^{\infty}{\alpha r \choose m} \mu^{m} \lambda^{\alpha r -m} \sum_{k=0}^{\infty} {\alpha r - m \choose k}(-1)^k\frac{1}{z^k}\\
& = \sum_{k=0}^{\infty}z^{-k}\left[(-1)^k\sum_{m=0}^{\infty}\mu^m\lambda^{-m}\sum_{r=0}^{\infty}\frac{(-t)^r}{r!}{\alpha r \choose m} {\alpha r -m \choose k} \right],
\end{align*}
the result follows by taking the coefficient of $z^{-k}.$
\end{proof}
\noindent By a standard conditioning argument, it follows that 
\begin{align*}
\mathbb{E}(N^{\alpha,\mu}(t)) = \mathbb{E}( \mathbb{E}(N(S_{\alpha,\mu}(t))|S_{\alpha,\mu}(t)) = \mathbb{E}(\lambda S_{\alpha,\mu}(t)) = \lambda \alpha \mu^{\alpha-1}t.
\end{align*}
Further,
\begin{align*}
\mathrm{Var}(N^{\alpha,\mu}(t)) &= \mathbb{E}(\mathrm{Var}(N(S_{\alpha,\mu}(t))|S_{\alpha,\mu}(t)) +  \mathrm{Var}(\mathbb{E}(N(S_{\alpha,\mu}(t))|S_{\alpha,\mu}(t))\\
& = \mathbb{E}(\lambda S_{\alpha,\mu}(t)) + \mathrm{Var}(\lambda S_{\alpha,\mu}(t))\\
& = \lambda \alpha \mu^{\alpha-1}t + \lambda^2 \alpha(1-\alpha) \mu^{\alpha-2}t.
\end{align*}
\begin{remark}
Using a similar argument as in Prop. \ref{prop5.1}, one can show that the two representations given in \eqref{tempered-SFPP} and \eqref{tempered-SFPP-Subodination} are equivalent.
\end{remark}

%

\subsection{The Tempered Time-Space-Fractional Poisson Process}
In this section we introduce and study tempered time-space-fractional Poisson process.
A subordination representation of tempered time-space-fractional Poisson process can be written as 
\begin{equation}\label{tempered-TSFPP-subordination}
N^{\alpha, \mu}_{\beta,\nu}: = N\left(S_{\alpha,\mu}(Y_{\beta,\nu}(t))\right) = N^{\alpha,\mu}(Y_{\beta,\nu}(t)),
\end{equation}
where $Y_{\beta,\nu}(t) = \inf\{r>0 : S_{\beta,\nu}(r) > t\}$ is the right-continuous inverse of tempered stable subordinator. Note that this process is non-Markovian due to the subordination component of $Y_{\beta,\nu}(t)$, which is not a Levy process. However, all the moments of this process are finite.\\

Alternatively, Taking a tempered fractional derivative in left side and tempered fractional shift operator in the right hand side of the equation \eqref{SFPP}, we obtained the governing difference-differential equation of the PMF of tempered time-space-fractional Poisson process. 
\begin{align}\label{tempered-TSFPP}
\frac{d^{\beta,\nu}}{dt^{\beta,\nu}}P^{\alpha, \mu}_{\beta,\nu}(k,t) = - ((\mu + \lambda(1-B))^{\alpha} - \mu^{\alpha}) P^{\alpha, \mu}_{\beta,\nu}(k,t),  \;\; \alpha\in(0,1], \;\;  \mu\geq 0,
\end{align}
where $\frac{d^{\beta,\nu}}{dt^{\beta,\nu}}$ is the Caputo tempered fractional derivative of order $\beta\in (0,1)$ with tempering parameter $\nu>0$.
The governing equation \eqref{tempered-TSFPP} reduces to the governing equation of SFPP by taking $\mu =\nu =0$ and $\beta =1$. The Riemann-Liouville tempered fractional derivative is defined by (see e.g. Alrawashdeh et al. 2016)
\begin{align*}
\mathbb{D}_t^{\beta, \nu} g(t) = e^{-\nu t} \mathbb{D}_t^{\beta}[e^{\nu t}g(t)] - \nu^{\beta}g(t),
\end{align*}
where
\begin{equation*}
\mathbb{D}_t^{\beta}g(t) = \frac{1}{\Gamma(1-\beta)}\frac{d}{dt}\int_{0}^{t}\frac{g(u)du}{(t-u)^{\beta}}
\end{equation*}
is the usual Riemann-Liouville fractional derivative of order $\beta\in(0,1)$. The Caputo derivative is defined by
$$
\frac{d^{\beta,\nu}}{dt^{\beta,\nu}}g(t) =  \mathbb{D}_t^{\beta}g(t) - \frac{g(0)}{\Gamma(1-\beta)}\int_{t}^{\infty}e^{-\nu r}\beta r^{-\beta-1}dr.
$$
The Laplace transform for the Caputo tempered fractional derivative for a function $g(t)$ satisfies
\begin{equation}
\mathcal{L}\left[\frac{d^{\beta,\nu}}{dt^{\beta,\nu}}g\right](s) = ((s+\nu)^{\beta}-\nu^{\beta})\tilde{g}(s) - s^{-1}((s+\nu)^{\beta}-\nu^{\beta}) g(0). 
\end{equation}
Further,
\begin{equation}\label{LT-Riemann-Liouville}
\mathcal{L}\left[\mathbb{D}_t^{\beta}g(t)\right](s) = ((s+\nu)^{\beta}-\nu^{\beta})\tilde{g}(s).
\end{equation}
Suppose $Y_{\beta, \nu}(t)$ be the right continuous inverse of tempered stable subordinator $S_{\beta,\nu}(t)$, defined by
\begin{equation}\label{def:itss}
Y_{\beta, \nu}(t)(t) = \inf\{y>0: S_{\beta, \nu}(y)> t\},~~ t\geq 0.
\end{equation}
The process $Y_{\beta, \nu}(t)(t)$ is called inverse tempered stable (ITS) subordinator.
A driftless subordinator $D(t)$ with L\'evy measure $\pi_D$ and density function $f$ has the L\'evy-Khinchin representation (see Bertoin, 1996)
 \begin{equation}
 \int_{0}^{\infty}e^{-ux}f_{D(t)}(x)dx = e^{-t\Psi_D(u)},
 \end{equation}
 where
 \begin{equation}\label{lsym}
 \Psi_D(u) = \int_{0}^{\infty}(1-e^{-uy})\pi_D(dy),~~ u>0,
 \end{equation}
 is called the Laplace exponent.
The L\'evy measure density  corresponding to a tempered stable subordinator is given by 
\begin{equation*}
\pi_{S_{\beta, \nu}}(u) = \frac{\beta}{\Gamma(1-\beta)}\frac{e^{-\nu
u}}{u^{\beta+1}}, \; u>0,
\end{equation*}
which satisfies the condition $\int_{0}^{\infty} \pi_{S_\beta, \nu}(u) du= \infty.$ Let $\mathcal{L}_{t\rightarrow s}(g(x,t)) = \bar{g}(x,s)$ be the Laplace transform (LT) of $g$ with respect to time variable $t$. Using Theorem 3.1 of Meerschaert and Scheffler (2008), the LT of  the density $h_{\beta, \nu}(x,t)$ of $Y_{\beta, \nu}(t)$ with respect to the time variable $t$ is given by
\begin{equation}\label{LT-ITSS}
\tilde{h}_{\beta, \nu}(x,s) = \frac{1}{s}\big((s+\nu)^{\beta}-\nu^{\beta}\big)e^{-x\big((s+\nu)^{\beta}-\nu^{\beta}\big)}.
\end{equation}
From \eqref{LT-ITSS}, it follows
\begin{align*}
\frac{\partial }{\partial x} \tilde{h}_{\beta, \nu}(x,s) = -\big((s+\nu)^{\beta}-\nu^{\beta}\big)\tilde{h}_{\beta, \nu}(x,s).
\end{align*}
Now by inverting the LT with the help of \eqref{LT-Riemann-Liouville}, it follows
\begin{equation}\label{ITSS-RL}
\frac{\partial }{\partial x}h_{\beta, \nu}(x,t) = - \mathbb{D}_t^{\beta,\nu} h_{\beta, \nu}(x,t).
\end{equation}
Further,
\begin{align}\label{Intermediate-LT}
- \frac{\partial }{\partial x} \tilde{h}_{\beta, \nu}(x,s) &= \left[\big((s+\nu)^{\beta}-\nu^{\beta}\big)\tilde{h}_{\beta, \nu}(x,s) - s^{-1}\big((s+\nu)^{\beta}-\nu^{\beta}\big)h(x,0)\right]\nonumber\\
&\hspace{2cm} + s^{-1}\big((s+\nu)^{\beta}-\nu^{\beta}\big)h(x,0).
\end{align} 
For inverting the LT in \eqref{Intermediate-LT} we will use the generalized Mittag-Leffler function, therefore we introduce it here.
The generalized Mittag-Leffler function, introduced by Prabhakar (1971), is defined by
\begin{equation}
M_{a,b}^c(z) = \sum_{n=0}^{\infty}\frac{(c)_n}{\Gamma(a n + b)}\frac{z^n}{n!},
\end{equation}
where $a, b, c \in \mathbb{C}$ with $\mathcal{R}(b) >0$ and $(c)_n$ is Pochhammer symbol see \eqref{Pochhammer}. When $c = 1$, it reduces to Mittag-Leffler function. Further,
\begin{equation}\label{gml-0}
M_{a,b}^c(0) = \frac{(c)_0}{\Gamma{(b)}} = \frac{1}{\Gamma{(b)}}.
\end{equation}
The function $F(s) = \frac{s^{ac - b}}{(s^a+ \eta)^c}$ has the inverse LT (see e.g. Kumar et al, 2018)
\begin{equation}\label{lt-for-special-func}
\mathcal{L}^{-1}[F(s)] = t^{b-1}M_{a,b}^{c}(-\eta t^a).
\end{equation}
Moreover, 
\begin{equation}\label{Ilt-for-special-func}
\mathcal{L}^{-1}\left[\frac{1}{s(s+\nu)^{-\beta}}\right] = t^{-\beta}M_{1,1-\beta}^{-\beta}(-\nu t),
\end{equation}
which follows by taking $a=1, b= 1-\beta, c=-\beta$ and $\eta =\nu$. Now by inverting the LT in \eqref{Intermediate-LT} with the help of \eqref{Ilt-for-special-func}, it follows
\begin{equation}\label{ITSS-fpde}
-\frac{\partial }{\partial x}h_{\beta, \nu}(x,t) =  \frac{\partial^{\beta,\nu}}{\partial t^{\beta,\nu}} h_{\beta, \nu}(x,t) + \left(t^{-\beta}M_{1,1-\beta}^{-\beta}(-\nu t) - \nu^{\beta} \right) \delta(x),
\end{equation}
where $h_{\beta, \nu}(x,0) = \delta(x)$ is the Dirac delta function.
\noindent Taking $\nu =0$ in \eqref{ITSS-fpde} and using \eqref{gml-0}, it follows
\begin{equation}
- \frac{\partial}{\partial x} h_{\beta,0}(x,t) = \frac{\partial^{\beta}}{\partial t^{\beta}}h_{\beta,0}(x,t) + \frac{t^{-\beta}}{\Gamma(1-\beta)}\delta(x),
\end{equation}
which is the governing equation of the density function of inverse $\beta$-stable subordiantor, which complements the result obtained in literature (see e.g. Meerschaert and Straka, 2013; Hahn et al. 2011).

\begin{proposition}\label{Caputo-TSTFPP}
The PMF of tempered time-space-fractional Poisson process defined in \eqref{tempered-TSFPP-subordination} satisfies \eqref{tempered-TSFPP}. 
\end{proposition}
\begin{proof}
Note that,
\begin{align}\label{intermediate-TemperedTSFPP}
P^{\alpha, \mu}_{\beta,\nu}(k,t) &= \mathbb{P}\left(N^{\alpha,\mu}(Y_{\beta,\nu}(t)) = k \right) = \mathbb{E} \left( \mathbb{P}\left(N^{\alpha,\mu}(Y_{\beta,\nu}(t)) = k |Y_{\beta,\nu}(t)\right) \right)\nonumber\\
& = \int_{0}^{\infty}P^{\alpha,\mu}(k,y)h_{\beta,\nu}(y,t)dy,
\end{align}
where $P_{\alpha,\mu}(k,t)$ is the PMF of tempered space-fractional Poisson process and $h_{\beta,\nu}(x,t)$ is the probability density function of inverse tempered stable subordinator.
Using \eqref{intermediate-TemperedTSFPP} and \eqref{ITSS-fpde}
\begin{align*}
\frac{d^{\beta,\nu}}{dt^{\beta,\nu}}P^{\alpha, \mu}_{\beta,\nu}(k,t) &= \int_{0}^{\infty}P^{\alpha,\mu}(k,y)\frac{d^{\beta,\nu}}{dt^{\beta,\nu}}h_{\beta,\nu}(y,t)dy\\
& = -\int_{0}^{\infty}P^{\alpha,\mu}(k,y)\frac{\partial }{\partial y}h_{\beta, \nu}(y,t)dy\\
&\hspace{1cm} - \left(t^{-\beta}M_{1,1-\beta}^{-\beta}(-\nu t) - \nu^{\beta} \right)\int_{0}^{\infty}P^{\alpha,\mu}(k,y) \delta(y))dy\\
& = - P^{\alpha,\mu}(k,y)h_{\beta, \nu}(y,t)|_{y=0}^{y=\infty} + \int_{0}^{\infty}\frac{d}{dy}P^{\alpha,\mu}(k,y)h_{\beta, \nu}(y,t)dy \\
&\hspace{1cm}- \left(t^{-\beta}M_{1,1-\beta}^{-\beta}(-\nu t) - \nu^{\beta} \right)P^{\alpha,\mu}(k,0)\\
& = \int_{0}^{\infty}\frac{d}{dy}P^{\alpha,\mu}(k,y)h_{\beta, \nu}(y,t)dy\\
& =  - ((\mu + \lambda(1-B))^{\alpha} - \mu^{\alpha})\int_{0}^{\infty}P^{\alpha,\mu}(k,y)h_{\beta, \nu}(y,t)dy\\
& = - ((\mu + \lambda(1-B))^{\alpha} - \mu^{\alpha})P^{\alpha, \mu}_{\beta,\nu}(k,t),
\end{align*}
using \eqref{tempered-SFPP} and the fact that $P^{\alpha,\mu}(k,0) =0,\; k>0.$
\end{proof}

\begin{remark}
Using similar argument as in Prop. \ref{Caputo-TSTFPP} with \eqref{ITSS-RL}, it follows that
\begin{equation*}
\mathbb{D}^{\beta,\nu}_t P^{\alpha, \mu}_{\beta,\nu}(k,t) = - ((\mu + \lambda(1-B))^{\alpha} - \mu^{\alpha})P^{\alpha, \mu}_{\beta,\nu}(k,t),\;k>0,\;t>0.
\end{equation*}
\end{remark}

\begin{proposition}
The state probabilities for tempered time-space-fractional Poisson process are given by
\begin{align}\label{PMF-tempered-TSFPP}
P^{\alpha,\mu}_{\beta,\nu}(k,t)
=(-1)^{k}e^{-t\nu} \sum_{m=0}^{\infty}{t}^{m}{\nu}^{m}\sum_{r=0}^{\infty}(-t^{\beta})^{r}M^{r}_{\beta,\beta r+m+1}({t^{\beta}}{\nu}^{\beta})\sum_{h=0}^{r}{r \choose h}(-{\mu}^{\alpha})^{r-h}\\\times \sum_{l=0}^{\infty}{\alpha h \choose l}{\alpha h-l \choose k}{\mu}^{l}{\lambda}^{\alpha h-l}, \; k=0,1,\ldots, \; \mu \geq 0, \; \nu \geq 0, \; t\geq0.
\end{align}
\end{proposition}
\begin{proof}
Suppose G(z,t) is the $z$-transform of $P^{\alpha,\mu}_{\beta,\nu}$, then
\begin{align*}
\frac{d^{\beta,\nu}}{dt^{\beta,\nu}}G(z,t) = - ((\mu + \lambda(1-z^{-1}))^{\alpha} - \mu^{\alpha})G(z,t),  \;\; \beta\in(0,1),\;\alpha\in(0,1], \;\;  \mu\geq 0, \; \nu\geq 0,
\end{align*}
using Laplace transform with respect to the time variable $t$ and assuming $|(\mu + \lambda(1-z^{-1}))^{\alpha} - \mu^{\alpha}| < |(s+\nu)^{\beta}-\nu^{\beta}|$, leads to
\begin{align*}
\mathcal{L}\left[G(z,t)\right]=&\frac{1}{s}\left(1+\frac{((\mu + \lambda(1-z^{-1}))^{\alpha} - \mu^{\alpha})}{(s+\nu)^{\beta}-\nu^{\beta}}\right)^{-1}\\
= &\sum_{r=0}^{\infty}(-1)^{r}\frac{((\mu + \lambda(1-z^{-1}))^{\alpha} - \mu^{\alpha})^{r}}{s((s+\nu)^{\beta}-\nu^{\beta})^{r}},
\end{align*}
suppose $F(s)=\frac{1}{(s^{\beta}-\nu^{\beta})^{r}}$ then the inverse LT of F(s) from \eqref{lt-for-special-func}, shifting property of LT $G(s)=F(s+\nu)$ then the inverse LT of $G(s)=e^{-\nu t}.t^{\beta r-1} M^{r}_{\beta,\beta r}(\nu^{\beta} t^{\beta})$,
\begin{align*}
\mathcal{L}^{-1}\left[\frac{G(s)}{s}\right]=\int_{0}^{t}e^{-\nu y} y^{\beta r-1}M^{r}_{\beta,\beta r}(\nu^{\beta} y^{\beta})dy 
\end{align*}
using the property from  Kilbas et al. (2004), 
\begin{align}
\int_{0}^{t}y^{\mu-1} M_{\rho,\mu}^{\nu}(w y^{\rho})(t-y)^{\nu-1}dy=\Gamma({\nu}) t^{\nu+\mu-1} M_{\rho,\mu+\nu}^{\nu}(w t^{\rho}).
\end{align}
Then $\mathcal{L}^{-1}{\left[\frac{G(s)}{s}\right]}=e^{-t \nu}\sum_{m=0}^{\infty}{\nu}^{m}t^{\beta r+m} M^{r}_{\beta,\beta r+m+1}(\nu^{\beta} t^{\beta})$, for simplicity we assume $H(t)=\mathcal{L}^{-1}\left[\frac{G(s)}{s}\right]$. Now $G(z,t)$
\begin{align*}
G(z,t)=&\sum_{r=0}^{\infty}(-1)^{r}\sum_{h=0}^{r}{r \choose h}(-{\mu}^{\alpha})^{r-h}\sum_{l=0}^{\infty}{\alpha h \choose l}{\mu}^{l}(\lambda(1-z^{-1}))^{\alpha h-l} H(t)\\
=&\sum_{r=0}^{\infty}(-1)^{r}\sum_{h=0}^{r}{r \choose h}(-{\mu}^{\alpha})^{r-h}\sum_{l=0}^{\infty}{\alpha h \choose l}{\mu}^{l}{\lambda}^{\alpha h-l} \sum_{k=0}^{\infty}(-1)^{k}{\alpha h-l \choose k} (z)^{-k} H(t)\\
=&\sum_{k=0}^{\infty}z^{-k}\left[(-1)^{k}e^{-t\nu}\sum_{m=0}^{\infty}{t}^{m}{\nu}^{m}\sum_{r=0}^{\infty}(-t^{\beta})^{r}M^{r}_{\beta,\beta r+m+1}({t}^{\beta}{\nu}^{\beta})\sum_{h=0}^{r}{r \choose h}(-{\mu}^{\alpha})^{r-h}\right.\\
&\hspace{1.5cm}\left.\times \sum_{l=0}^{\infty}{\alpha h \choose l}{\alpha h-l \choose k}{\mu}^{l}{\lambda}^{\alpha h-l}\right],
\end{align*}
the result follows by taking the coefficient of $z^{-k}.$
\end{proof}
\begin{remark}
For $\mu=0,\; \nu=0$, eq. \eqref{PMF-tempered-TSFPP}  is equivalent to putting $m=0,\; l=0$ and $r=h$, which reduces to
\begin{equation}
P^{\alpha, 0}_{\beta, 0}(k,t) = (-1)^k \sum_{r=0}^{\infty} (-1)^r{\alpha r \choose k}\frac{\lambda^{\alpha r} t^{\beta r}}{\Gamma(\beta r+1)},
\end{equation}
 whic is same as the PMF of time-space-fractional Poisson process given in\eqref{TSPP}.
\end{remark}

\subsection{Fractional Equation with Gegenbauer Type Fractional Operator and Generalized Poisson Distributions}
In this section, we introduce new class of fractional differential equations and their solutions.\\ 
We consider the backward-shift fractional operator
\begin{align*}
\nabla_u^{d}&=(1-2uB+B^{2})^{d}=(1-2cos(\nu)B+B^2)^{d}\\
&=[(1-e^{i\nu}B)(1-e^{-i\nu}B)]^{d}\\
 &= -\lambda^{2d} (1-2uB+B^{2})^{d} P_{d}^u(k,t),\; |u|\leq1,\; d\in(0,1/2],
\end{align*}
which often appears in the study of the so-called Gegenbaurer times series (see Beran                                             1994, p. 213, Gray et al. 1989 or Espejo et al. 2014 ) \\
Note that for $u=1$ the fractional operator $\nabla_u^{d}=(1-B)^{d}$ reduces to \eqref{STFPP} with $\alpha=2d\in(0,1)$.
where $u=cos(\nu)$ or $\nu =cos^{-1}(u)$. We introduce the following fractional equation for unknown function $P_d^u(k,t), t\geq 0, k = 0,1,\ldots$ 
\begin{align}\label{Gegenbauer-DDE}
\frac{d}{dt}P_d^u(k,t)= -\lambda^{2d}\nabla_u^{d} \left(P_d^u(k,t)\right) 
 &= -\lambda^{2d} (1-2uB+B^{2})^{d} P_{d}^u(k,t), k >0, \; d\in(0,1/2],\\
\frac{d}{dt}P_d^u(0,t)&=-\lambda^{2d}P^{u}_{d}(0,t),
\end{align}
with initial conditions
\begin{equation}\label{initial-conditions1}
  P_{d}^u(k,0) = \delta_{k,0}.
\end{equation}
Using the $z$-transform in both side, it follows
\begin{align*}
 \frac{d}{dt}\left[\mathcal{Z}{P_{d}^u(k,t)}\right]=-\lambda^{2d}[\mathcal{Z}\{{(1-2uB+B^{2})^dP_{d}^u(k,t)\}}].
\end{align*}
Expanding the fractional difference operator as 
$$(1-2uB+B^{2})^d= \sum_{j=0}^{\infty}\sum_{k=0}^{\infty}(-1)^{j+k}{d \choose j}{d \choose k}(e^{i\nu})^{j}(e^{-i\nu})^{k}B^{j+k},$$ 
leads to
\begin{align}\label{intermediate-eq1}
     \frac{d}{dt}\left[\mathcal{Z}P{_{d}^u(k,t)}\right]=\left[-\lambda^{2d}\left(1-\frac{2u}{z}+\frac{1}{z^{2}}\right)^d\right] \mathcal{ Z}{P_{d}^u(k,t)}.
\end{align}
Now solve the equation \eqref{intermediate-eq1} for $\mathcal{Z}{P_{d}^u(k,t)}$, we obtain
$$
 \mathcal{ Z}{P_{d}^u(k,t)} = A e^{-{\lambda}^{2d}\left(1-\frac{2u}{z}+\frac{1}{z^{2}}\right)^{d} t}.
 $$
 Using initial conditions in \eqref{initial-conditions1}, it follows that $A=1.$\\
\begin{equation}\label{PGF}
\mathcal{ Z}{P_{d}^u(k,t)}=  e^{-{\lambda}^{2d}\left(1-\frac{2u}{z}+\frac{1}{z^{2}}\right)^{d} t}\\
\end{equation}
\begin{align*}
 \mathcal{ Z}{P_{d}^u(k,t)}& =\left[1+\frac{(-\lambda)^{2d} t}{1!}\left\{1+ \frac{1-2uz}{z^{2}}\right\}^{d} +\ldots +\frac{(-\lambda)^{2kd} t^k}{k!}\left\{1+ \frac{1-2uz}{z^{2}}\right\}^{kd}+\ldots\right] \\
 &=1+\frac{(-\lambda)^{2d} t}{1!}\left\{\sum_{k=0}^{\infty}{d \choose k}\frac{\sum_{n=0}^{k}{k \choose n}(-2uz)^{n}}{z^{2k}}\right\}+\ldots\\
&\hspace{1cm}+\frac{(-\lambda)^{2kd} t^{k}}{k!}\left\{\sum_{k=0}^{\infty}{kd \choose k}\frac{\sum_{n=0}^{k}{k \choose n}(-2uz)^{n}}{z^{2k}}\right\}+\ldots\\
\end{align*}
With the help of coefficient of $z^{-k}$, the inverse $z$-transform gives
\begin{align*}
P_{d}^u(k,t)&=\sum_{r=0}^{\infty}\frac{(-\lambda)^{2rd}t^r}{r!}\left\{{rd \choose p}\cdot {p\choose 0} + (2u)^{2}{rd \choose p+1}\cdot {p+1\choose 2} + \ldots \right.\\ 
& \left. + (2u)^{2p-2}{rd \choose 2p-1}\cdot {2p-1\choose 2p-2} + 2^{2p}{\alpha \choose 2p}\cdot {2p\choose 2p}\right\},\; d\in(0, 1/2],\; k=2p.
\end{align*}
We have the following proposition.
\begin{proposition}
Solution of the initial value problem \eqref{Gegenbauer-DDE} is of the form
\begin{align}\label{Gegenbauer-PMF}
P_{d}^u(k,t)&=\sum_{r=0}^{\infty}\frac{(-\lambda)^{2rd}t^r}{r!}\left\{{rd \choose p}\cdot {p\choose 0} + (2u)^{2}{rd \choose p+1}\cdot {p+1\choose 2} + \ldots \right.\nonumber\\ 
& \left. + (2u)^{2p-2}{rd \choose 2p-1}\cdot {2p-1\choose 2p-2} + 2^{2p}{\alpha \choose 2p}\cdot {2p\choose 2p}\right\},\; d\in(0, 1/2],\; k=2p,
\end{align}
\end{proposition}
\begin{remark}
Taking $\alpha =2d, k=2p$ in \eqref{space-fractional-PMF} and $u=1$ in \eqref{Gegenbauer-PMF} both the results coincides. 
\end{remark}
\noindent Further, we also introduce the following Gegenbauer type space-time-fractional Poisson process by replacing the integer order derivative in \eqref{Gegenbauer-DDE} by a fractional derivative of order $\beta$, as follows
\begin{align}\label{Gegenbauer-Space-Time-DDE}
\frac{d^\beta}{dt^\beta}Q_{d,\beta}^u(k,t) &= - \lambda^{d}\nabla_u^{d} Q_{d,\beta}^u(k,t)\nonumber\\
& =(-\lambda)^{2d}(1-2uB+b^{2}B)^{d} Q_{d,\beta}^u(k,t), \; d\in(0,1/2], \; \beta\in(0, 1], k=1,2,\ldots\\
 \frac{d^\beta}{dt^\beta}Q_{d,\beta}^u(0,t) &=({-\lambda}^{2d}Q^{u}_{d,\beta}(0,t)\;\;\mathrm{with}\;\; Q_{d}^{u,\beta}(k,0) = \delta_{k,0}.
\end{align}
Using a similar approach, we have the result.
\begin{proposition}
The PMF for the Gegenbauer type space-time-fractional Poisson process defined by the difference-differential equation in \eqref{Gegenbauer-Space-Time-DDE} is
\begin{align*}
Q_{d,\beta}^u(k,t)&=\sum_{r=0}^{\infty}\frac{(-\lambda)^{2rd}t^{r\beta}}{\Gamma(1+2p\beta)}\left\{{rd \choose p}\cdot {p\choose 0} + (2u)^{2}{rd \choose p+1}\cdot {p+1\choose 2} + \ldots \right.\\ 
& \left. + (2u)^{2p-2}{rd \choose 2p-1}\cdot {2p-1\choose 2p-2} + 2^{2p}{\alpha \choose 2p}\cdot {2p\choose 2p}\right\},\;   d\in(0,1/2].
\end{align*}
\end{proposition}
Note that if $u\neq d$ there is no stochastic process which gives the equations \eqref{Gegenbauer-DDE}. To see this we consider the $z$-transform of $P^u_d(k,t)$
\begin{align}
  G(z,t)=E[z^{- N_d^u(t)}] = \sum_{r=0}^{\infty}{P_{d}^u(k,t)}z^{-r}.
\end{align}
By normalization axiom of probability, it is necessary that $G(1) =1.$
From equation \eqref{PGF}
$$
G(z,t) = e^{{-\lambda}^{2d}\left(1-{\frac{2u}{z}}+\frac{1}{z^{2}}\right)^{d} t}.
$$
With the condition of normalization
$$
G(1,t)=e^{{-2}{\lambda}^{2d}{(1-u)^{d}t}},\;  |u|\leq 1.
$$
It is easy to see that $G(1) < 1$, for all $u$ except the case when $u=1$. So the distribution $P_{d}^u(k,t)$ will not satisfy the normalization condition. One can say that $N_d^u(t)$ is a defective random variable that mean there is some positive mass concentrated at $\infty.$\\

\noindent One can also consider shift operators of the form,

\begin{equation}
\frac{d}{dt} P^*(k,t) = -\lambda[(1-B)^{\alpha_1} + (1-B)^{\alpha_2}]P^*(k,t),
\end{equation}
with initial condition $P^*(k,0) = \delta_{k,0}$. It is easy to show that
\begin{equation}
P^*(k,t) = (-1)^k\sum_{r=0}^{\infty}(-1)^r\frac{\lambda^r}{r!}\sum_{m=0}^{r}{r\choose m} {\alpha_1m+ \alpha_2(r-m) \choose k},\;k=0,1,\ldots
\end{equation}
Similar to the Gagenbauer shift operator case, the function $P^*$ may not be a probability distribution.

\section{Conclusion}
In this article, we introduce and study tempered time-space-fractional Poisson processes, which may provide more flexibility in modeling of real life data. Further, we argue that $z$-transform is more useful than the probability generating function in solving the difference-differential equations since it is more general and hence may be used in the situations where the solution is not a probability distribution indeed. To support this, we work with the Gegenbauer type fractional shift operator. Our results generalize and complements the results available on time- and space-fractional Poisson processes.\\

\noindent {\bf Acknowledgments:}   
N. Leonenko was supported in particular by Australian Research Council's Discovery Projects funding scheme (project DP160101366)and  by project MTM2015-71839-P of MINECO, Spain (co-funded with FEDER funds).

\vspace{2cm}
\noindent {\bf \Large References}
\noindent
\begin{namelist}{xxx}
\item{} Alrawashdeh, M.S., Kelly, J.F., Meerschaert, M.M., Scheffler, H.-P.: Applications of inverse tempered stable subordinators. Comput. Math. Appl. 73, 89--905 (2016).

\item{} Aletti, G., Leonenko, N.,  Merzbach, E.: Fractional Poisson fields and martingales, Journal of Statistical Physics, 170, 700--730 (2018).

\item{}  Bingham, N.H.:  Limit theorems for occupation times of Markov processes, Z. Wahrscheinlichkeitstheorie verw. Geb. 17, 1 22 (1971).

\item {} Beghin, L., Orsingher, E.: Fractional Poisson processes and related random motions. Electron. J. Probab., 14. 1790--1826 (2009).

\item{} Beran, J.: Statistics for Long-Memory Processes, Chapman \& Hall, New York 1994.

\item{} Bertoin, J.: L\'evy Processes, Cambridge University Press, Cambridge, 1996.

\item{}    Buchak, K. V., Sakhno, L. M.: On the governing equations for Poisson and Skellam processes time-changed by inverse subordinators, Theory of Probability and Mathematical Statistics, 98, 87--99 (2018a).

\item{}  Buchak, K. V., Sakhno, L. M.: Properties of Poisson processes directed by compound Poisson-Gamma subordinators, Modern Stochastics: Theory and Applications, 5, 167--189 (2018b).

\item{} Espejo, R.M., Leonenko, N., Ruiz-Medina, M.D.: Gegenbauer random fields. Random Oper. Stoch. Equ. 22, 1--16 (2014).
%

\item{} Gray, H.L., Zhang, N.F., Woodward, W.A.: On generalized fractional processes, J. Time Series Analysis, 10, 233--258 (1989).

\item{} Gorenflo, R., Kilbas, A.A., Mainardi, F., Rogosin, S.V.: Mittag -Leffler Functions, Related Topics and Applications, Berlin  2014.

\item{}  Hahn, M.G., Kobayashi, K., Umarov,  S.: Fokker-Planck-Kolmogorov equations associated with time-changed fractional Brownian motion. Proc. Amer. Math. Soc., 139, 691--705 (2011).

\item{} Hosking, J.R.M.: Fractional Differencing. Biometrika 68, 165--176 (1981).

\item{} Kilbas, A.A., Saigo, M. and Saxena, R. K.: Generalized Mittag-Leffler function and generalized fractional calculus operators, Integral Transforms and Special Functions, vol. 15, no. 1, pp. 31--49 (2004).

\item{} Kumar, A., Gajda, J., Wylomanska, A. Poloczanski, R.: Fractional Brownian Motion Delayed by Tempered and Inverse Tempered Stable Subordinators, Method. Comp. Appl. Probab. (Forthcoming) 2018.

\item{} Kumar, A., Vellaisamy, P.: Inverse tempered stable subordinators, Statistics \& Probability Letters, 103, 134--141 (2015).

\item {} Laskin, N.: Fractional Poisson process. Commun. Nonlinear Sci. Numer. Simul. 8, 201--213 (2003).

\item{} Laskin, N.: Some applications of the fractional Poisson probability distribution
J. Math. Phys., 50, 113513 (2009).

\item{}  Mainardi, F.: Fractals and Fractional Calculus in Continuum Mechanics. Springer Verlag 1997.

\item{} Mainardi, F.,  Gorenflo, R., Scalas, E.: A fractional generalization of the Poisson processes, Vietnam Journal of Mathematics, Vol. 32, pp. 53--64 (2004).

\item{} Meerschaert, M.M., Nane, E., Vellaisamy, P.: The fractional Poisson process and the inverse stable subordinator. Electron. J. Probab., 16, 1600--1620 (2011).

\item{}  Meerschaert, M.M., Scheffler, H.: Triangular array limits for continuous time random walks, Stochastic Process. Appl. 118, pp. 1606--1633 (2008).

\item{} Meerschaert, M.M., Sikorski, A.: Stochastic Models for Fractional Calculus, De Gruyter studies in Mathematics, vol-43 (2012). 

\item{} Meerschaert, M.M., Straka, P.: Inverse stable subordinators. Math Model Nat Phenom 8, 1--16 (2013).

\item{} Orsingher, E., Polito, F.: The space-fractional Poisson process, Statistics \& Probability Letters,82, 852--858 (2012).

\item{} Prabhakar, T.R.: A singular integral equation with a generalized Mittag-Leffler function in the kernel. Yokohama Math J, 19, 7--15 (1971).

\item{}  Uchaikin, V.V. and Sibatov, R.T.: 
A fractional Poisson process in a model of dispersive charge transport in semiconductors,
Russian J. Numer. Anal. Math. Modelling, 23, 283--297 (2008).

\end{namelist}                       
\end{document}